\documentclass[utf8,12pt,english]{aclart}
 
\usepackage{smfhref}

\usepackage{lmodern}

\theoremstyle{plain}
\newtheorem*{theo*}{\theoname}
\theoremstyle{remark}
\newtheorem{remas}[subsection]{Remarks}

\def\abs#1{\left|#1\right|}
\def\norm#1{\left\|#1\right\|}
\def\CAP{{\mathrm{cap}}}
\def\ddc{\mathop{\mathrm d\mathrm d^c}}
\def\C{\mathbf C}
\def\R{\mathbf R}
\def\Q{\mathbf Q}
\def\Z{\mathbf Z}
\let\ra\rightarrow
\let\liminf\varliminf
\let\limsup\varlimsup
\let\phi\varphi \let\eps\varepsilon

\def\ie{\emph{ie.}\xspace}
\def\eg{\emph{eg.}\xspace}

\def\ord{\operatorname{ord}}
\def\flim{\operatornamewithlimits{fine.lim}}

\begin{document}

\title[The Theorem of Jentzsch--Szeg\H o on an analytic curve]
 {The Theorem of Jentzsch--Szeg\H o on an analytic curve. 
  Application to the irreducibility of truncations of power series}
\author{Antoine Chambert-Loir}
\begin{altabstract}
Le th\'eor\`eme de Jentzsch-Szeg\H o d\'ecrit la mesure
limite d'une suite de mesures discr\`etes
associ\'ee aux z\'eros d'une suite convenable de polyn\^omes en une variable.
Suivant la pr\'esentation que font Andrievskii et Blatt
dans~\cite{andrievskii-blatt2002}, on \'etend ici
ce r\'esultat aux surfaces de Riemann compactes,
puis aux courbes analytiques sur un corps ultram\'etrique.
On donne pour finir
quelques corollaires du cas particulier 
de la droite projective sur un corps ultram\'etrique
à l'irréductibilité des polynômes-sections d'une série entière
en une variable.
\end{altabstract}
\begin{abstract}
A theorem of Jentzsch--Szeg\H o describes the limit measure of a sequence
of discrete measures associated to zeroes of a sequence of polynomials
in one variable. 
We extend this theorem to compact Riemann surfaces and 
to analytic curves over ultrametric fields. 
This theory is applied to the problem of irreducibility
of truncations of power series with coefficients in ultrametric fields.
\end{abstract}
\maketitle

The construction of families of irreducible polynomials 
as truncations of power series with rational coefficients
has attracted the attention of many mathematicians, 
\emph{e.g.,} Schur~\cite{schur1930}, Coleman~\cite{coleman1987}, and others.
A basic example of this phenomenon is given by the exponential
function, 
\[ \exp (T)= \sum_{j=0}^\infty \dfrac1{j!}T^j, \]
all of which truncations 
\[ f_n(T) = \sum_{j=0}^n \dfrac1{j!}T^j \]
are irreducible over~$\Q$.

In his class (autumn~2009) at Princeton University, N.~Katz asked
whether this was a general phenomenon, \emph{i.e.,} for general conditions
on the power series which would imply irreducibility of all truncations.
Alternatively, he asked for conditions which would imply reducibility.
He referred to a theorem of Jentzsch~\cite{jentzsch1916}
in complex analysis
according to which any point of the circle of convergence is a limit
point of zeroes of these truncations, 
provided the radius of convergence is finite and positive.

More generally, Szeg\H o~\cite{szego1922} proved that the
probability measures defined by zeroes of a suitable subsequence
of truncations are equidistributed on the circle of convergence.
In particular, these truncations cannot be all split over~$\R$,
let alone over~$\Q$.

Today, these theorems are understood in the context of
potential theory on the Riemann sphere
(see, \emph{e.g.}, 
the book of Andrievskii and Blatt~\cite{andrievskii-blatt2002}).

This paper was prompted by the fact that an appropriate $p$-adic analogue
of the Jentzsch--Szeg\H o theorem imply stronger
irreducibility properties of truncations of power series
whose \emph{$p$-adic} radius of convergence is finite and positive.
As a corollary of our main theorem (Theorem~\ref{theo.2})
we obtain the following result
 (see Theorem~\ref{exem-3.4}). We first recall the definition of
the (generalized) Tate algebras: for any positive real number~$R$,
$ K\{ R^{-1}T\} $ is the subalgebra of $K[[T]]$ consisting
of power series $\sum a_j T^j$ such that $\abs{a_j}R^j\ra 0$;
it is the algebra of holomorphic functions on the closed disk 
$E(0,R)=\{\abs T\leq R\}$ in the sense of Berkovich.

\begin{theo*}
Let $p$ be a prime number, $K$ a finite extension of~$\Q_p$,
$R$ a positive real number
and $f=\sum_{j=0}^\infty a_j T^j\in K\{R^{-1} T\}$.
For any nonnegative integer~$n$, 
let $f_n(T)=\sum_{j=0}^n a_jT^j$ be the truncation of~$f$ in degree~$n$.

Then, for any positive integer~$d$ and 
any subsequence~$(n_k)$ such $\abs{a_{n_k}}^{1/n_k}\ra 1/R$,
the number of $K$-irreducible factors of~$f_{n_k}$ of degree~$\leq d$
is $\mathrm o(n_k)$. In particular, the largest
degree of an irreducible factor of~$f_{n_k}$ tends to infinity
for $k\ra\infty$.
\end{theo*}

The classical example of the exponential series (which
however does not belong to the Tate algebra $\Q_p\{\abs{p}^{-1/(p-1)}T\}$)
indicates that
one cannot hope for much more in general. Indeed, the theory of
Newton polygons implies that for $f=\exp(T)$, $f_n$
has irreducible factors over~$\Q_p$ of degrees $p,p(p-1),\dots,
p^{m-1}(p-1)$, where $m$ is the largest integer
such that $p^m\leq n$.

Observe also that the existence of a subsequence~$(n_k)$
as in the Theorem implies that the radius of convergence
of~$f$ is equal to~$R$.

In the proofs, the restriction
to elements of a Tate algebra is essential. 
We explain in Remark~\ref{rema.defect} why this
is a major defect for the application to irreducibility. 
However, an easy construction (Example~\ref{contrex})
shows that the theorem
does not extend to arbitrary power series.

\medskip

This article has three parts.
First we generalize the methods of Andrieveskii and Blatt
to include compact Riemann surfaces of arbitrary genus,
see Theorem~\ref{theo.1}. The main interest of this extension
is to prepare the second part where we prove
an analogue of the Jentzsch--Szeg\H o theorem in the ultrametric
setting, \emph{i.e.,} when the compact Riemann surface is replaced by
a smooth projective analytic curve in the sense 
of Berkovich~\cite{berkovich1990}.
The non-archimedean potential theory developed by Thuillier~\cite{thuillier2005}
and Baker/Rumely~\cite{baker-rumely2010} 
is formally identical to the classical complex potential theory.
In particular, the proof of Section~1 applies verbatim.
In Section~3, we
apply this to the Berkovich projective line and deduce
our main results concerning irreducible factors
of truncations of power series over a locally compact ultrametric field.

\medskip

\paragraph*{Acknowledgments}
This paper was written during a stay 
at the Institute for Advanced Study in Princeton. 
I am very grateful to Nick Katz for asking this question,
which revived my interest in potential theory on ultrametric curves, 
and to Enrico Bombieri for many discussions on related questions.
The idea of extending the theorem of Jentzsch--Szeg\H o to curves of
arbitrary genus is due to Antoine Ducros; I would like to thank him,
Amaury Thuillier, Matt Baker, and Robert Rumely for
many stimulating exchanges concerning potential theory. 
I would also like to thank the referee for his/her
careful reading.

I was partially supported by Institut universitaire de France,
and the National Science Foundation grant DMS-0635607.

\section{Riemann surfaces}

Let $M$ be a compact connected Riemann surface.
Let $p$ be a point of~$M$ and $E$ a compact non-polar subset of~$M$;
we assume that $\Omega=M\setminus E$ is connected and contains~$p$.
Fix a local parameter~$z$ in a neighborhood of~$p$.
The Green function~$G_E$ is the unique subharmonic function
on~$M\setminus\{p\}$ such that
\begin{enumerate}
\item it vanishes on~$E$ (up to a polar subset
of~$\partial E$), 
\item it is harmonic on $M\setminus(\{p\}\cup E)$ 
\item and it has an expansion around~$p$ of the form
\[ G_E(q)=\log\abs {z(q)^{-1}/\CAP(E)}+\mathrm o(1) .\]
\end{enumerate}
The positive real number~$\CAP(E)$ is the capacity of~$E$
with respect to~$p$, relative to the local parameter~$z$.
More intrinsically, there exists a norm~$\norm{\cdot}^\CAP$ on
the complex tangent line~$T_pM$ such that $\CAP(E)$ is the norm of
the tangent vector~$\partial/\partial z\in T_pM$.
This norm does not depend on the choice of the local parameter~$z$.
The equilibrium measure of~$E$ is the probability measure
\[ \mu_E = \ddc G_E + \delta_p ;\]
it is supported on the boundary~$\partial E=E\setminus\mathring E$
of~$E$.
Finally, for any function~$f$
on~$M$, we define $\norm{f}_E=\sup_{E}\abs f$. If $f$
is holomorphic on a neighborhood of~$E$, 
then $\norm f_E=\sup_{\partial E}\abs f$ (maximum principle).

Let $k$ be a positive integer
and $f\in\Gamma(M,\mathscr O(kp))$,
a meromorphic function on~$M$, holomorphic on~$M\setminus\{p\}$
with a pole of order~$\leq k$ at~$p$.
Its leading coefficient at~$p$, $j^k(f)$, is defined as
\[ j^k(f) = \lim_{q\ra p} f(q) z(q)^k \left( \frac\partial{\partial z}\right)^{\otimes k}; \]
it is an element of~$T_pM^{\otimes k}$,
independent of the choice of the local parameter~$z$,
and vanishing if and only if the order of the pole of~$f$ at~$p$ is~$<k$.

\begin{lemm}\label{lemm.1}
Let $k$ be a positive integer and let $f$ be any non-zero element
of~$\Gamma(M,\mathscr O(kp))$.
The function $\frac1k\log\abs f-G_E$ is subharmonic on~$\Omega$.
For any point $q\in M\setminus\{p\}$, one has 
\[ \abs{f(q)}\leq \norm{f}_E \exp(k G_E(q)).\]
In particular, \[ \norm{j^k(f)}^\CAP \leq \norm{f}_E .\]
\end{lemm}
\begin{proof}
Set $\phi=\frac 1k\log\abs{f}- G_E$.
The function~$\phi$ is subharmonic on~$\Omega\setminus\{p\}$, 
since on this set, $\log\abs f$ is subharmonic and $G_E$ is harmonic.
In fact, it is subharmonic on~$\Omega$ since, 
after choosing a local parameter~$z$ at~$p$, 
$q\mapsto f(q)z(q)^k$ is holomorphic in a neighborhood of~$p$.
By the maximum principle for subharmonic functions
of~\cite{tsuji1968} (Theorem~III.28, p.~77),
we have
\[ \sup_\Omega\phi = \sup_{q\in \partial\Omega}
    \flim_{\substack{z\ra q\\ z\in \Omega}} \phi(z)
= \sup_{q\in\partial\Omega} \frac1k\log\abs{f(q)} = \frac1k\log \sup_{\partial E}\abs f
= \frac1k\log \norm f_E. \]
(Taking limits for the fine topology, we may ignore the eventual
polar subset of~$\partial E$ at which $G_E$ does not tend to~$0$.)
Moreover, 
\begin{align*} \phi(p) &= \frac1k\lim_{q\ra p} \phi(q)=\frac1k\lim_{q\ra p} \log \abs{f(q) z(q)^k}
   -   \lim_{q\ra p} \left( G_E(q)-\log\abs{z(q)}^{-1}\right)\\
&= \frac1k\log \norm{j^k(f)}^\CAP  .
\end{align*}
Consequently, $\norm{j^k(f)}^\CAP\leq\norm f_E$.
\end{proof}

For such a function~$f$, let $\nu(f)$ be the measure
$f^*\delta_0/k$ given by the zeroes of~$f$ (divided by~$k$).
It is a positive measure on~$M$ with total mass~$\leq1$, and
a probability measure if and only if $j^k(f)\neq0$.

\begin{theo}\label{theo.1}
Let $(k_n)$ be a sequence of positive integers. For any~$n$, let
$f_n\in\mathscr O(k_n p)$ be a non-zero meromorphic function
on~$M$ with a pole of order at most~$k_n$ at~$p$.
Assume that:
\begin{enumerate}
\item $\displaystyle \limsup_n \frac1{k_n}\log\norm{f_n}_E \leq 0$;
\item for any compact subset~$C$ in~$\mathring E$,  $\nu(f_n)(C)\ra 0$;
\item there exists a non-empty compact subset~$S$ in~$\Omega$ such that
\[ \liminf_n  \sup_S (\frac1{k_n}\log\abs{f_n}-G_E) \geq0.\]
\end{enumerate}
Then, the sequence of measures~$(\nu(f_n))$ converges to the equilibrium
measure~$\mu_E$ in the weak-$*$ topology.
\end{theo}

\begin{remas}\label{remas-1.3}
\begin{enumerate}\def\theenumi{\alph{enumi}}\def\labelenumi{\theenumi)}
\item
For any~$n$, the function
 $\phi_n=k_n^{-1}\log\abs{f_n}-G_E$ is subharmonic on~$\Omega$. 
In particular, it is upper-semicontinuous, 
hence bounded from above on any compact subset of~$\Omega$.
The upper-bound on~$S$ in condition~(3) is therefore finite.
More precisely, we have seen that
 $\sup_S \phi_n\leq k_n^{-1}\log\norm{f_n}_E$.
Condition~(3) implies that 
$\liminf_n k_n^{-1}\log\norm{f_n}_E\geq0$.
Condition~(1) implies  $\limsup_n \sup_S\phi_n \leq 0$. 
The conjunction of Conditions~(1) and~(3) is thus equivalent
to the two equalities
\[ \lim_n \frac1{k_n}\log\norm{f_n}_E = \lim_n \sup_S \phi_n = 0. \]
\item
For $S=\{p\}$, Condition~(3) is equivalent to 
\[ \liminf_n \frac1{k_n}\log\norm{j_{k_n}(f_n)}^\CAP\geq 0 \]
but requiring this inequality is more restrictive.
For example, if $f_n\in\mathscr O((k_n-1)p)$, then $j^k(f_n)=0$
but Condition~(3) still can be valid for some compact subset.
When Condition~(3) holds for $S=\{p\}$, 
Condition~(1) implies that 
\[  \lim\frac1{k_n}\log\norm{j_{k_n}(f_n)}^\CAP=0.\]
\end{enumerate}
\end{remas}

\begin{lemm}\label{lemm.minor-T}
Let $(k_n)$ be a sequence of positive integers. For any $n$,
let $f_n\in \mathscr O(k_n p)$. Assume that Conditions~(1) and~(3)
of Theorem~\ref{theo.1} hold. Then, for any compact
and non-polar subset $T\subset \Omega$, one has
\[ \liminf_n \sup_T(\frac1{k_n}\log\abs{f_n}-G_E) = 0.\]
\end{lemm}
\begin{proof}
Set $\phi_n=\frac1{k_n}\log\abs{f_n}-G_E$ and
let $S$ be a non-empty compact subset of~$\Omega$ such that 
$\liminf_n \sup_S\phi_n\geq 0$. 
Let $m=\liminf_n\sup_T\phi_n$.  
By Remark~\ref{remas-1.3}, a), 
it suffices to prove that $m\geq0$.

First assume that $T$ is disjoint from~$S$.
Then there exists a harmonic function~$u$ on $\Omega\setminus T$
which, up to a set of capacity zero, vanishes on the boundary of~$E$
and equals~$m$ at the boundary of~$T$.
Let $\eps>0$; by Remark~\ref{remas-1.3}, a), 
for sufficiently large integers~$n$, 
$\phi_n\leq u+\eps$ on~$\partial E$ (Condition~(1)), as well as
on~$\partial T$ (by the definition of~$m$), modulo subsets of zero capacity.
Since $\phi_n$ is subharmonic on~$\Omega\setminus T$
the maximum principle of~\cite{tsuji1968} (Theorem~III.28, p.~77)
implies  that $\phi_n\leq u+\eps$ on~$\Omega\setminus T$.
Therefore, $\sup_S \phi_n \leq \sup_Su+\eps$
and $\liminf_n \sup_S\phi_n \leq \sup_S u+\eps$. 
Considering arbitrary small positive~$\eps$, we obtain
$\liminf_n \sup_S\phi_n\leq \sup_Su$.
If $m<0$, the strong maximum principle implies that 
$u<0$ on~$\Omega\setminus T$ (since $\Omega$
is connected, the closure of any connected component of~$\Omega\setminus T$ 
meets~$T$), hence $\sup_S u<0$, a contradiction.

In general, let $T'$ be a compact non-polar
subset of~$\Omega$, disjoint from~$S\cup T$; 
for example, a closed disk (of non-empty interior) contained
in the complementary subset.
By the previous case,  the statement holds for~$T'$
(since $T'$ is disjoint from~$S$). Since $T$ is disjoint
from~$T'$, it also holds for~$T$.
\end{proof}

\begin{lemm}
Under the assumptions of Theorem~\ref{theo.1}, the sequence~$(k_n)$
converges to~$+\infty$.
\end{lemm}
\begin{proof}
Assume otherwise. Choosing a subsequence if necessary, we may
assume that the sequence~$(k_n)$ is constant, equal to a
positive integer~$k$.
Then, $\lim\log\norm{f_n}_E =0$; in particular, the
sequence~$(\norm {f_n}_E)$ is bounded.
Since $E$ is infinite (being non-polar), $\norm{\cdot}_E$ is a norm
on~$\Gamma(M,\mathscr O(k p))$. Since this space is finite-dimensional,
all norms on it are equivalent, and the sequence~$(f_n)$
contains a converging subsequence. Its limit is a function $f\in\mathscr O(kp)$.
The convergence is uniform on any compact subset of~$M\setminus\{p\}$.
By Condition~(2) and Hurwitz's Theorem, $f$ does not vanish on~$\mathring E$.

Let $S$ be a compact and non-polar subset of~$\Omega\setminus\{p\}$; 
By Lemma~\ref{lemm.minor-T}, we have
\[ 0\leq \liminf_n \sup_S\left( \frac 1{k_n}\log\abs{f_n}-G_E\right)
 = \sup_S \left(\frac1k \log\abs f-G_E\right). \]
In particular, $\sup_S \abs f\geq e^{k G_E}\geq 1$.
Since $S$ is arbitrary, we conclude that 
$\abs{f(q)}\geq e^{k G_E(q)}\geq 1$ 
for any $q\in \Omega\setminus\{p\}$. This implies
that $f$ doesn't vanish on~$\Omega\setminus\{p\}$ and that the order
of its pole at~$p$ is equal to~$k$.
Letting the point~$q$ tend to a point of~$\partial E$,
we see that $\abs{f}\geq 1$ on~$\partial E$.

In conclusion, $f$ doesn't vanish on~$M\setminus\{p\}$, which 
contradicts the presence of a pole at~$p$. 
\end{proof}

\begin{lemm}\label{lemm.vadh}
Under the assumptions of Theorem~\ref{theo.1}, any limit measure~$\nu$
of the sequence~$(\nu(f_n))$ is a probability measure supported on~$\partial E$.
\end{lemm}
\begin{proof}
First of all, $\nu$ is a probability measure.
Indeed,
let $d_n$ be the order of the pole of~$f_n$ at~$p$; it is the smallest
integer~$d$ such that $f\in\mathscr O(dp)$ and the mass of~$\nu(f_n)$
equals $d_n/k_n$. Set
\[ \phi'_n=\frac1{k_n}\log\abs{f_n}-\frac{d_n}{k_n}G_E. \]
This is a subharmonic function on~$\Omega$, bounded from above by  $\frac1{k_n}\log\norm{f_n}_E$.
Moreover, since
\[ \phi_n = \frac1{k_n}\log\abs{f_n}-G_E
=\phi'_n - \frac{k_n-d_n}{k_n} G_E, \]
one has
\[ \sup_S \phi_n \leq \sup_S \phi'_n - \frac{k_n-d_n}{k_n} \inf_S G_E
\leq \frac1{k_n}\log\norm{f_n}_E  - \frac{k_n-d_n}{k_n} \inf_S G_E ,\]
where $S$ is any compact and non-polar subset of~$\Omega$, disjoint from~$p$,
so that \mbox{$\inf_S G_E>0$.}
When $n\ra\infty$, it follows that
\[ 0 \leq\liminf_n\sup_S \phi_n \leq - \limsup_n \frac{k_n-d_n}{k_n}\inf_S G_E.\]
Consequently, $\limsup_n\frac{k_n-d_n}{k_n}\leq 0$, hence
\[ \lim_n \frac{d_n}{k_n}=1. \]

We now show that the support of~$\nu$ is contained in~$\partial E$.
By Condition~(2), it is disjoint from~$\mathring E$.
Thus it suffices to prove that it is contained in~$E$.

Let $C$ be a compact subset in~$\Omega$. We claim that $\nu(f_n)(C)\ra0$.
Indeed, for $c\in C$, 
let $G_{E,c}$ be the Green function for~$E$ with pole at~$c$.
For any integer~$n$, let $(c_{n,j})_{j\in J_n}$ be the family of those zeroes of~$f_n$ which belong to~$C$, repeated according to their multiplicity.
Set
\[ \phi'_n = \phi_n + \frac1{k_n}\sum_{j\in J_n} G_{E,c_{n,j}};\]
it is a subharmonic function on~$\Omega$, bounded from above 
by~$\frac1{k_n}\log\norm{f_n}_E$.
Let $S$ be a compact and non-polar subset of~$\Omega$, disjoint from~$C$.
Since the function $(q,c)\mapsto G_{E,c}(q)$ is continuous and positive
on $S\times C$,
it follows that $\inf_S \inf_{c\in C} G_{E,c}>0$.
Then,
\[ \sup_S \phi'_n  \geq \sup_S \phi_n + \inf_S \frac1{k_n}\sum_{j\in J_n} G_{E,c_{n,j}} \geq \sup_S \phi_n + \nu(f_n)(C) \inf_{S}\inf_{c\in C} G_{E,c},\]
and 
\[ \liminf_n \sup_S \phi'_n \geq \liminf_n \nu(f_n)(C)\inf_{c\in C} G_{E,c}.\]
On the other hand,
the inequality $\phi'_n\leq \frac1{k_n}\log\norm{f_n}_E$
implies that
\[ \liminf_n \sup_S \phi'_n \leq 0.\]
It follows that $\liminf_n \nu(f_n)(C)=0$, as claimed.
Passing to subsequences, this implies that $\nu(\mathring C)=0$.
Since $\Omega$ is locally compact, we conclude that
the support of~$\nu$ is disjoint from~$\Omega$, so is contained in~$E$.
This concludes the proof.
\end{proof}

\begin{proof}[Proof of Theorem~\ref{theo.1}]
Since $M$ is compact, the space of probability measures on~$M$
is also compact.
It suffices to prove that $\mu_E$ is the only possible limit value 
of the sequence~$(\nu(f_n))$.
Let $\nu$ be such a limit value.
By Lemma~\ref{lemm.vadh}, $\nu$ is a probability measure supported on~$\partial E$. 
Replacing~$k_n$ by the order of the pole of~$f_n$ at~$p$,
we suppose that for any~$n$, 
$f_n\not\in \mathscr O((k_n-1)p)$, in other words, $j^k(f_n)\neq0$.

Since~$\nu$ admits a countable basis of neighborhoods, and 
after passing to a subsequence,
we may assume that $\nu(f_n)$ converges to~$\nu$.

Let $g(\cdot,\cdot)_p=$  be a Green kernel on~$M\times M$
relative to the point~$p$; this is a distribution on~$M\times M$
satisfying the following properties:
\begin{itemize}
\item  the partial differential relation
 $\ddc g+\delta_{\Delta}=\delta_{p\times M}+\delta_{M\times p}$
holds;
\item the distribution~$g$ is symmetric.
\end{itemize}
If $M=\mathbf P^1$, $p=\infty$ and $M\setminus\{p\}$
is identified with~$\C$, 
one can take $g(z_1,z_2)_p=\log\abs{z_1-z_2}^{-1}$.
We refer to~\cite{rumely1989} for a proof of 
the existence of such a distribution in general.
(In the notation of that book, 
this is the distribution $-\log[\cdot,\cdot]_p$.)
Moreover, for any points~$m$ and~$m'$ in~$M\setminus\{p\}$, one has
\begin{equation}
 \label{eq.vanish}
\lim_{z\ra p} \left( g(z,m)_p-g(z,m')_p \right) = 0 , \end{equation}
uniformly when~$m$ and~$m'$ belong to a fixed compact subset of~$M\setminus\{p\}$. The Green kernel thus defines a local parameter at~$p$, well-defined
up to multiplication by a local holomorphic function 
of absolute value equal to~$1$ at~$p$.

For any measure~$\alpha$ on~$M$ whose support does not contain~$p$,
let \[ U^\alpha = \int g(z,\cdot)_p \mathrm d\alpha(z) \]
be the potential of~$\alpha$ with respect to the kernel~$g$;
it is a distribution on~$M$ such that 
$\ddc U^\alpha+\alpha=\norm\alpha \delta_p$, where
$\norm\alpha=\langle \alpha,1\rangle$ is the total mass of~$\alpha$.
In particular, $U^\alpha$ is subharmonic outside of the support of~$\alpha$.
If $\alpha$ is positive, $U^\alpha$ is subharmonic outside of~$p$.
If the total mass of~$\alpha$ is zero then 
Equation~\eqref{eq.vanish} implies that
$U^\alpha$ is continuous and vanishes at~$p$.
The Green function~$G_E$ with pole at~$p$ can be written in terms
of the potential~$U^{\mu_E}$ of the equilibrium measure~$\mu_E$:
one has $U^{\mu_E}=-G_E-\log\CAP (E)$.

For any meromorphic function $f\in\mathscr O(kp)$
such that $j^k(f)\neq 0$, $\log\abs f+k U^{\nu(f)}$
is a harmonic function on the compact space~$M$, hence constant.
Let $a$ be a complex number such that this function equals $\log\abs a$.
Then,
\[ \phi=\frac1k\log\abs f-G_E=-U^{\nu(f)}-G_E+\frac1k\log\abs a
= -U^{\nu(f)}+U^{\mu_E}+\frac1k \log\norm{j^k(f)}^\CAP \]
since $U^{\nu(f)}-U^{\mu_E}$ vanishes at~$p$.

To prove Theorem~\ref{theo.1}, we
establish the inequality
$U^\nu\leq U^{\mu_E}$. Then a standard argument shows that
$\mu_E=\nu$.

\medskip

Let $V$ a compact non-polar subset of~$\Omega$ and $W$
a compact neighborhood of~$V$ contained in~$\Omega$.
Assume that $p\not\in V$ but $p\in W$.
The measure $\nu(f_n)$ splits canonically as the sum 
 $\nu_W(f_n)+\nu^W(f_n)$
of two measures, where $\nu_W(f_n)=\nu(f_n)\mathbf 1_W$
is supported on~$W$, while $W$ has measure~$0$ with
respect to $\nu^W(f_n)=\nu(f_n)(\mathbf 1-\mathbf 1_W)$.

According to Lemma~\ref{lemm.vadh},
$\norm{\nu_W(f_n)}$ tends to~$0$ when~$n$ goes to~$+\infty$,
so that $\nu$ is also a limit value of the sequence~$(\nu^W(f_n))$.
Since $g(\cdot,\cdot)_p$ is bounded on 
$\complement W\times V$,
$U^\nu$ is a limit value of the sequence~$(U^{\nu^W(f_n)})$, 
for the topology of uniform convergence on~$V$.
(For any subset~$A$ of~$M$, we write $\complement M$
to denote the complementary subset $M\setminus A$ to~$A$ in~$M$.)


Following Andrievskii and Blatt~\cite{andrievskii-blatt2002}, let
us decompose
\begin{align*}\phi_n & =\frac1{k_n}\log\abs{f_n}-G_E \\ 
&= -U^{\nu(f_n)}+U^{\mu_E} +\frac1{k_n}\log\norm{j^k(f_n)}^\CAP  \\
& = \left( -U^{\nu^W(f_n)}+\norm{\nu^W(f_n)} U^{\mu_E}\right) 
 + \norm{\nu_W(f_n)} U^{\mu_E} 
  \\ & \qquad\qquad {}
   + \left( \frac1{k_n}\log\norm{j^k(f_n)}^\CAP - U^{\nu_W(f_n)}\right).
\end{align*}
The function $-U^{\nu_W(f_n)}$ is subharmonic on~$M\setminus\{p\}$.
Let $R$ be any compact neighborhood of~$p$, 
disjoint from~$V$ and contained in~$W$.
Since $V\subset\complement R\subset M\setminus\{p\}$,
the maximum principle implies that
\[ \sup_V (-U^{\nu_W(f_n)}) \leq \sup_{\complement R} (-U^{\nu_W(f_n)})
 =\sup_{\partial R}(-U^{\nu_W(f_n)}) .
\]
Consequently,
\begin{multline*} \sup_V  \left( \frac1{k_n}\log\norm{j^k(f_n)}^\CAP -U^{\nu_W(f_n)} \right) \\
\leq \sup_{\partial R} \phi_n 
 - \inf_{\partial R} \left( -U^{\nu^W(f_n)}+\norm{\nu^W(f_n)} U^{\mu_E}\right)
 - \norm{\nu_W(f_n)}\inf_{\partial R}  U^{\mu_E} .\end{multline*}

The function $-U^{\nu^W(f_n)}+\norm{\nu^W(f_n)} U^{\mu_E}$
is continuous on~$R$, and harmonic on~$\mathring R$. 
For $n$ going to infinity,  it converges uniformly
to~$-U^{\nu}+U^{\mu_E}$, which is again continuous on~$R$,
harmonic on~$\mathring R$ and vanishing at~$p$.
It follows that
\[ \lim_n \inf_{\partial R} \left( -U^{\nu^W(f_n)}+\norm{\nu^W(f_n)} U^{\mu_E}\right) = \inf_{\partial R} \left(-U^{\nu}+U^{\mu_E}\right) .\]
By Lemma~\ref{lemm.1} and Condition~(1) of Theorem~\ref{theo.1},
$\phi_n\leq \frac 1{k_n}\log\norm{f_n}_E$
and \[ \limsup_n \sup_{\partial R} \phi_n \leq 0. \]
Since $\norm{\nu_W(f_n)}$ tends to~$0$,
\[ \liminf_n \norm{\nu_W(f_n)}\inf_{\partial R}  U^{\mu_E} =0. \]
Finally,
\[ \limsup_n \sup_V  \left( \frac1{k_n}\log\norm{j^k(f_n)}^\CAP  -U^{\nu_W(f_n)} \right)
\leq \sup_{\partial R} \left( U^\nu-U^{\mu_E}\right). \]
Choose  the compact neighborhood~$R$ arbitrarily close to~$\{p\}$.
Since $U^\nu-U^{\mu_E}$ is continuous and vanishes at~$p$
we deduce
\begin{equation}\label{eq.jkf-u}
 \limsup_n \sup_V  \left( \frac1{k_n}\log\norm{j^k(f_n)}^\CAP - U^{\nu_W(f_n)}\right) \leq 0 .\end{equation}

Furthermore,
\begin{align}
\notag
 \inf_{V} U^{\nu^W(f_n)} &= \inf_V \left(U^{\nu(f_n)}-U^{\nu_W(f_n)}\right)\\
\notag
&=\inf_V \left( \left( U^{\nu(f_n)}-U^{\mu_E} \right)
 - \left(U^{\nu_W(f_n)}-U^{\mu_E}\right)\right)\\
\notag \label{1.7}
& \leq \inf_V 
 \left( U^{\nu(f_n)}-U^{\mu_E} \right) - \inf_V U^{\nu_W(f_n)}+ \sup_V U^{\mu_E}.
\end{align}
Since $U^{\nu^W(f_n)}$ converges uniformly to~$U^\nu$
on~$V$, Equation~\eqref{eq.jkf-u} implies that
\[ \inf_V U^\nu \leq \limsup_n \inf_V \left(
  U^{\nu(f_n)}-U^{\mu_E} -\frac1{k_n}\log\norm{j^{k_n}(f_n)}^\CAP \right)
               + \sup_V U^{\mu_E}. \]
However
\[  U^{\nu(f_n)}-U^{\mu_E} -\frac1{k_n}\log\norm{j^{k_n}(f_n)}^\CAP
= -\frac1{k_n}\log\abs{f_n} + G_E, \]
so that
\[ \limsup_n \inf_V \left(
  U^{\nu(f_n)}-U^{\mu_E} -\frac1{k_n}\log\norm{j^{k_n}(f_n)}^\CAP \right)
\leq - \liminf_n \sup_V \left( \frac1{k_n}\log\abs{f_n}-G_E\right)
\leq 0 \]
by Lemma~\ref{lemm.minor-T}.

This proves  the inequality
\[ \inf_V U^\nu \leq \sup_V U^{\mu_E}. \]
Since $U^\nu$ and $U^{\mu_E}$ are continuous on~$\complement (E\cup\{p\})$,
this implies that $U^\nu\leq U^{\mu_E}$ on~$\complement(E\cup\{p\})$.

It follows that $\nu=\mu_E$. Indeed, since $U^\nu$ is subharmonic
outside~$E$ and $U^{\mu_E}$ is bounded from above by~$-\log\CAP (E)=I(\mu_E)$
on~$\partial E$, up to a polar subset, we have
$U^\nu\leq I(\mu_E)$ on~$\partial E$. Then,
the energy~$I(\nu)=\int U^\nu\,\mathrm d\nu$ of~$\nu$
is bounded from above by~$I(\mu_E)$. Since~$\mu_E$ is the unique measure 
of minimal energy supported on~$E$,
we obtain that $\nu=\mu_E$, as claimed.
\end{proof}

\section{Analytic curves over ultrametric fields}

Let $K$ be a complete ultrametric valued field (of any characteristic).
Let $M$ a smooth projective, geometrically connected curve over~$K$;
let $p\in M(K)$ and let $z$ be a local parameter at~$p$.

We view~$M$ as a $K$-analytic curve in the sense 
of Berkovich~\cite{berkovich1990}. 
Recall
that $M\setminus\{p\}$ is the Berkovich spectrum $\mathscr M(A)$
of the $K$-algebra $A=\Gamma(M\setminus\{p\},\mathscr O_M)$,
\ie, the set of multiplicative seminorms on this $K$-algebra
which extend the absolute value of~$K$, endowed
with the coarsest topology for which all maps
$a\mapsto (x\mapsto \abs{a}(x))$ are continuous.
We use the standard notation in this subject: if $x\in\mathscr M(A)$
and $a\in A$, $\abs a(x)$ is the value at~$a$ of the semi-norm~$x$.
Every $K$-rational point of~$M$ defines a canonical element of~$M$;
if $q\neq p$, this is just the semi-norm $a\mapsto \abs{a(q)}$ on~$A$.
By~\cite{berkovich1990},
the space~$M$ is connected, locally contractible and compact.
If $K$ admits a countable dense subset, the space~$M$ is also metrizable.

By the works of Favre/Jonsson~\cite{favre-jonsson2004},
Favre/Rivera-Letelier~\cite{favre-rl2006},
Thuillier~\cite{thuillier2005}, 
Baker/Rumely~\cite{baker-rumely2010}, it is well-known
that such a space admits a potential theory formally analogous
to that on compact Riemann surfaces. Therefore, all statements
of the first Section, and their proofs, translate
directly to the ultrametric setting.

When $M$ is the projective line, the required theory is 
the subject of the book~\cite{baker-rumely2010}
by Baker and Rumely.
In his unpublished PhD Thesis, Thuillier~\cite{thuillier2005} 
developed a more general theory, valid for arbitrary curves.
Here, we recall briefly the main aspects of his theory.

The Berkovich space~$M$ carries two sheaves,
the sheaf~$\mathscr A$ of smooth functions,
and its subsheaf~$\mathscr H$ of harmonic functions;
both are subsheaves of the sheaf of real valued continuous functions
on~$M$. 
There is a notion of subharmonic functions;
these obey a maximum principle. If $U$ is an open subset of~$M$
and $f\in\mathscr O(U)$ is an analytic function on~$U$,
the function $\log\abs f$ is subharmonic, and is harmonic
if $f$ doesn't vanish. Harmonic functions satisfy Harnack's principle;
in particular, a uniform limit of harmonic functions on an open
set is harmonic.

Furthermore, one defines 
sheaves of smooth forms, distributions~$\mathscr D^0$,
and currents~$\mathscr D^1$,
as well as a Laplace operator
$\ddc\colon\mathscr D^0\ra\mathscr D^1$. 
Smooth forms are locally finite linear combinations of Dirac measures
at points of type~II or~III,
distributions are dual to smooth forms, currents are
dual to smooth functions; 
there are canonical inclusions of smooth functions
into distributions, and of smooth forms into currents.

If $u$ is smooth, $\ddc u$ is a smooth form.
A function~$u$ on an open subset~$U$ of~$M$ is subharmonic
if and only if $\ddc u$ is a positive measure on~$U$.
Furthermore, a Radon measure~$\mu$ on~$M$ is of the form $\ddc T$,
for some distribution~$T$, if and only if $\langle\mu,1\rangle=0$.

The Dirichlet problem on an open subset of~$M$ is solved using
Perron's method. Barriers exist at any point
of~$M$ which is not of type~I. This implies existence
and uniqueness of a Green function~$G_E$ for a compact subset~$E$, 
with a pole at a prescribed $K$-point~$p\not\in E$.

In classical potential theory, or Abstract potential theory
(see, \eg, \cite{ransford95} and \cite{brelot1967})
 the \emph{kernel} plays an important rôle. It is
an upper-continuous function on the product space
$(M\setminus p)\times (M\setminus p)$.
For $M=\mathrm P^1$ and $p=\infty$, 
this is the so-called \emph{Hsia kernel} $\delta(\cdot,\cdot)_\infty$,
or rather $-\log\delta(\cdot,\cdot)_\infty$. Baker and Rumely
describe it in detail in~\cite[Chapter~4]{baker-rumely2010}.
In general, Thuillier sketches a construction of this kernel 
in Chapter~5 of~\cite{thuillier2005}:
for $m,m'\in M\setminus p$, $g(m,m')=g_m(m')$, where $g_m$
is the unique continuous function on~$M$, with values
in~$\R\cup\{\pm\infty\}$, solution of the equation
$\ddc g_m=\delta_m-\delta_p$
which admits an expansion 
\[ g_m(q)=\log\abs{z(q)}+\mathrm o(1) \]
in a neighborhood of~$p$.
This  function~$g$ is symmetric, continuous with respect to each variable,
lower semi-continuous, and even continuous outside the diagonal.
Moreover, this kernel~$g$ is the largest semi-continuous extension
of the kernel $-\log[\cdot,\cdot]_p$ constructed by Rumely 
in his book~\cite{rumely1989}.

If $\mu$ is a measure with compact support in~$M\setminus p$,
its potential~$U^\mu$ is the unique solution
of the distribution equation
\[ \ddc U^\mu=\mu-\langle\mu,1\rangle\delta_p \]
satisfying 
\[ U^\mu(q) = \langle\mu,1\rangle \log \abs{z(q)}+\mathrm o(1) \]
in a neighborhood of~$p$.
It can be computed using the kernel, by the formula
\[ U^\mu (m)=\int_{M\setminus p}  g(m,m')\,\mathrm d\mu(m'). \]

For $M=\mathrm P^1$, the maximum and continuity principles,
analogs to theorems of Maria and Frostman,
are proved by Baker and Rumely
(\cite{baker-rumely2010}, Theorems~6.15 and~6.18).
In general,
one can refer to Abstract potential theory.
By~\cite{brelot1967}, the maximum and continuity principles
are satisfied as soon as subharmonic functions
satisfy the maximum principle, which is the case.
(Note that Brelot's axiomatic in~\cite{brelot1967} 
only considers positive kernels. 
However, since we will only look at measures whose support is compact
in~$M\setminus p$, the required assertions remain true,
with essentially the same proofs.)

The energy of a measure~$\mu$ 
with  compact support in~$M\setminus p$
is given by the formula
\[ I(\mu)=\int_{M\setminus p} U^\mu(m)\,\mathrm d\mu(m) = \int_{(M\setminus p)^2}
g(m,m')\,\mathrm d\mu(m)\,\mathrm d\mu(m'). \]
Robin's constant~$V_p(E)$ of a compact subset~$E$ of~$M\setminus p$
is the lower bound of the energies of probability measures
supported on~$E$.
If $E$ is not polar, that is, if $V_p(E)\neq+\infty$,
there exists a unique probability measure~$\mu_E$ supported
on~$E$ such that $I(\mu_E)=V_p(E)$: this is the \emph{equilibrium measure}
of~$E$. The existence of equilibrium measures is a consequence
of compactness of the space of probability measures on~$E$. 
For $M=\mathrm P^1$, uniqueness
is shown in Prop.~7.21 of~\cite{baker-rumely2010},
relying on a \emph{strong maximum principle} (Prop.~7.17). 
Theorem~3.6.11 in Thuillier's~\cite{thuillier2005}
furnishes the {\og Evans functions\fg} used by
Baker and Rumely in their proof, so that existence and uniqueness of
an equilibrium measure holds in general.

In Section~1,
we had to extract converging subsequences of sequences of probability measures.
This is still possible when the field~$K$ admits a countable dense
subset since, in that case, the space~$M$ and the space of probability measures on~$M$ are compact and metrizable.
In the general case, subsequences may not suffice
but it suffices to carry out the arguments using
ultrafilters instead of subsequences. Alternatively,
one can also replace sequences by nets, 
as Baker and Rumely do in~\cite{baker-rumely2010}.

It is now clear that the arguments given in Section~1 to prove
Theorem~\ref{theo.1} translate in the present setting of analytic
curves over ultrametric fields and furnish the following theorem.

\begin{theo}\label{theo.2}
Let $K$ be a complete valued ultrametric field,
$M$ a projective smooth and geometrically connected curve over~$K$, 
viewed as a $K$-analytic curve in the sense of Berkovich.
Let $p$ be a $K$-rational  point of~$M$, $z$ 
a local parameter at~$p$. Let $E$ be a compact non-polar subset
of~$M\setminus\{p\}$ such that $\Omega=M\setminus E$ is connected.

For any non-zero rational function~$f$ on~$M$, 
let $\nu(f)$ be the probability measure on~$M$ given by
\[ \nu(f) = \frac1{\deg(f)} \sum_{f(q)=0} \ord_q(f) \delta_q, \]
where the sum is over the zeroes of~$f$;
we also set $\norm{f}_E=\sup_{\partial E}\abs f$.

Let $(k_n)$ be a sequence of positive integers. For any~$n$, let
$f_n\in\mathscr O(k_n p)$ be a \emph{non-zero} meromorphic function
on~$M$ having a pole of order at most~$k_n$ at~$p$.
Let us make the following assumptions:
\begin{enumerate}
\item $\displaystyle \limsup_n \frac1{k_n}\log\norm{f_n}_E \leq 0$;
\item for any compact subset~$C$ in~$\mathring E$,  $\nu(f_n)(C)\ra 0$;
\item there exists a non-empty compact subset~$S$ in~$\Omega$ such that
\[ \liminf_n  \sup_S (\frac1{k_n}\log\abs{f_n}-G_E) \geq0.\]
\end{enumerate}
Then, the sequence of measures~$(\nu(f_n))$ converges to the equilibrium
measure~$\mu_E$ for the weak-$*$ topology.
\end{theo}

\section{Applications to irreducibility}

Let $K$ be a complete ultrametric valued field (of any characteristic).
Let $M=\mathrm P^1$ be the projective line over~$K$ in the
sense of Berkovich. Let $p\in M$ be its point at infinity.
The space $M\setminus p$ is the analytic affine line $\mathrm A^1$,
that is, the Berkovich spectrum~$\mathrm M(K[T])$
of the polynomial algebra~$K[T]$.
Let us fix $T^{-1}$ as a local parameter at~$p$.

Let $R$ be a positive real number. The closed disk, denoted $E(0,R)$
by Berkovich, is the set of points~$x$
in~$\mathrm A^1$ such that $\abs{T}(x)\leq R$. 
It is a compact subset in~$\mathrm A^1$ whose Shilov boundary
has a unique point~$\xi(R)$;
in other words, any holomorphic function on~$E(0,R)$ reaches its maximum at~$\xi(R)$; this point is the Gauß seminorm 
\[ P = \sum_{j=0}^\infty a_j T^j \mapsto \max_j \abs{a_j}R^j; \]
the multiplicativity of this seminorm is Gauß's theorem.
We also write $\norm\cdot_R$ for the supremum norm of a polynomial
or an analytic function on the disk~$E(0,R)$.
The interior of~$E(0,R)$ in the affine line~$\mathrm A^1$
is equal to~$E(0,R)\setminus\{\xi(R)\}$ (\cite{berkovich1990}, Corollary~2.5.13).
The open disk $D(0,R)$ is the set of points~$x$ such that $\abs{T}(x)<R$.

The Green function for~$E(0,R)$ (with pole at infinity)
is given by $x\mapsto \log \max(\abs{T}(x)/R,1)$;
its equilibrium measure is the Dirac measure at~$\xi(R)$.
As a particular case of Theorem~\ref{theo.2}, we obtain:
\begin{prop}\label{prop.3}
Let us consider a sequence of polynomials~$(f_n)$ satisfying
the following properties:
\begin{enumerate}
\item the degree~$k_n$ of~$f_n$ tends to~$+\infty$;
\item the sequence~$(f_n)$ converges uniformly on the disk~$E(0,R)$
to a non zero function.
\item the sequence $(a_n)$ given by the leading coefficient~$a_n$ 
of~$f_n$ satisfies $\lim \abs{a_n}^{1/k_n}\ra 1/R$.
\end{enumerate}
Then, the sequence $(\nu(f_n))$ of probability measures
converges to the Dirac measure at the point~$\xi(R)$.
\end{prop}
\begin{proof}
Let $f$ be the limit of~$f_n$; it is an analytic function on
the disk~$E(0,R)$, hence an element of the Tate algebra $K\{R^{-1}T\}$.
Condition~(1) of Theorem~\ref{theo.2} is obviously verified.

Since $f\neq 0$ and $E(0,R)$ is compact and connected,
the function~$f$ has only finitely many zeroes on~$E(0,R)$, counted
with multiplicities. Analogously to Hurwitz's theorem
in complex analysis, Condition~(2) of Theorem~\ref{theo.2}
also holds. Indeed, let us even show that $\nu(f_n)(E(0,R))\ra 0$.
Up to replacing~$K$ by a complete
algebraically closed extension,
we may assume that $R=\abs a^{-1}$ for some $a\in K^*$.
Then, the theory
of Newton polygons implies that $k_n\nu(f_n)(C)$
is the degree of the reduction of the polynomial
$\widetilde{f_n(aT)}$. Clearly this degree converges
to that of the polyonomial $\widetilde{f(aT)}$, so we obtain
that $\nu(f_n)(E(0,R))\ra 0$.

Finally, Condition~(3) also holds, with $S=\{\infty\}$.
This implies that $\nu(f_n)$ converges to the Dirac measure at~$\xi(R)$,
as claimed.
\end{proof}

\begin{rema}\label{rema.defect}
In the complex setting, it would be sufficient to assume that
the sequence~$(f_n)$ converges uniformly on compact
subsets of the open disk of radius~$R$,
while in the $p$-adic case, we have to assume that 
the uniform convergence holds on the full closed disk.
This discrepancy is due to the fact that 
the interior of the $p$-adic unit~$E=E(0,R)$
disk is much larger than the open $p$-adic unit disk~$D(0,R)$.
In fact, $\mathring E$ is the complement to the Gauss point~$\xi(R)$
in~$E$. This makes  the equidistribution statement 
of Theorem~\ref{theo.2} almost pointless in this particular case.
Indeed, the easiest part of its proof shows that any limit measure
is supported by~$E$. And since its assumption~(2) 
requires that any limit measure does not charge~$\mathring E$,
this forces the limit measure to be a Dirac mass at the Gauss
point~$\xi(R)$. This is however unavoidable, \emph{cf.} Example~\ref{contrex}.
\end{rema}

The following corollaries are especially interesting under
the supplementary assumption that the coefficients of
the polynomials~$f_n$
belong to a locally compact subfield~$K_0$ of~$K$.
They can be proved directly for elements of a Tate algebra $K\{R^{-1}T\}$
but we keep to our initial goal
and view them as a consequence of the behavior
of the limit measures of zeroes established in Proposition~\ref{prop.3}:
they apply for any sequence~$(f_n)$ for which the sequence~$(\nu(f_n))$
converges to the Dirac measure at a Gau\ss\ point~$\xi(R)$.

\begin{coro}
Let $K_0$ be a locally  compact subfield of~$K$.
When $n\ra\infty$, 
the number of $K_0$-roots of the polynomial~$f_{n}$ 
is $\mathrm o(k_n)$.
\end{coro}
\begin{proof}
The conclusion is that
 $\mu_{n}(K_0)\ra 0$. If it didn't hold, the limit measure
of the sequence~$(\mu_n)$ would charge~$K_0$.
But $\xi(R)\not\in K_0$.
\end{proof}

\begin{coro}\label{coro-3.2}
Let $K_0$ be a locally  compact subfield of~$K$.
Let $d$ be a positive integer. When $n\ra\infty$,
the number of irreducible factors in~$K_0[T]$ of the polynomial~$f_n$
whose degree is~$\leq d$ is $\mathrm o(k_n)$.
\end{coro}
\begin{proof}
Replacing $K$ by the completion of an algebraic closure, we assume
that it is algebraically closed.
Let $K_d\subset K$ be the extension of~$K_0$ (in a fixed algebraic closure
of~$K$) generated by all roots 
of all polynomials of degree~$\leq d$ in $K_0[T]$.
It is well known that $K_d$ is a finite extension of~$K_0$.
In particular, it is a locally compact subfield,
hence the result follows from the first corollary.
\end{proof}

\begin{coro}\label{coro-3.3}
Let $K_0$ be a locally  compact subfield of~$K$ and 
assume that the polynomials~$f_n$ belong to~$K_0[T]$.
When $n\ra\infty$, the maximal degree of an irreducible factor of~$f_n$
tends to~$+\infty$.
\end{coro}
\begin{proof}
Otherwise, up to replacing the sequence~$(f_n)$ by a subsequence of it,
the irreducible factors of~$f_n$ would have a uniformly bounded degree,
which contradicts the previous corollary.
\end{proof}

Let us explicit the particular case where, for each integer~$n$,
the polynomial~$f_n$ is the degree~$n$ truncation 
of a fixed power series $f=\sum_j a_jT^j$ with coefficients
in a locally compact $p$-adic field.
\begin{theo}\label{exem-3.4}
Let $K$ be a finite extension of~$\Q_p$ and let 
$f=\sum_{j=0}^\infty a_j T^j$ be a power series with coefficients
in~$K$.
Let $R=(\limsup_j \abs{a_j}^{1/j})^{-1}$ be its radius of convergence.
Let us assume that $0<R<\infty$ and that $f\in K\{R^{-1}T\}$.

For each integer~$n$, let $f_n=\sum_{j=0}^n a_jT^j$.
Let $(n_k)_{k\geq 0}$ be a increasing sequence of integers such that
$\abs{a_{n_k}}^{1/n_k}$ tends to~$1/R$ when $k\ra\infty$.
Then, the number of irreducible factors
of~$f_{k_n}$ of bounded degree is negligible before~$n_k$,
and the maximum degree of an irreducible factor of~$f_{n_k}$
tends to infinity when $k\ra\infty$.
\end{theo}

We conclude this paper by the following promised construction,
which shows that in hypothesis~(2) of Proposition~\ref{prop.3},
the closed disk~$E(0,R)$ cannot be replaced by the open disk~$D(0,R)$,
and that Theorem~\ref{exem-3.4} does not hold for arbitrary
power series of radius of convergence~$R$.

\begin{exem}\label{contrex}
We begin with the simple following observation: 
for any degree~$m$ polynomial $f\in\Z[T]$,
and any integer $n>m+1$, 
there exists a unique monic polynomial~$F\in\Z[T]$ 
of degree~$n$
such that $F\equiv f\mod {T^{m+1}}$
and $F\equiv 0 \mod {(T-1)^{n-m-1}}$.
Indeed, write $F=f+T^{m+1}g(T-1)+T^n$, 
where the unknown polynomial~$g\in\Z[T]$
has degree~$\leq n-m-2$; the condition on~$F$ translates into
the condition that $g(T)$ is congruent modulo~$T^{n-m-1}$ to
the power series with integer coefficients
given by the expansion of $-(f(1+T)+(1+T)^n)/(1+T)^{m+1}$.
The existence and uniqueness of~$F$ follows at once.

Thanks to this observation, we may construct
by induction a sequence~$(F_n)$ of monic polynomials
with integer coefficients such that 
$d_n=\deg(F_n)=2^{n+1}-2$ such that 
$F_{n+1}(T)\equiv F_n (T)\mod {T^{d_n-1}}$
and $F_n(T)$ vanishes at order at least~$2^n-1$ at $T=1$.
   
It follows that there exists a power series~$f$
such that, for any integer~$n\geq 0$, the polynomial~$F_n$ 
is the truncation in degree~$d_n$ of~$f$.

Fix a prime number~$p$ and view the power series~$f$
as a power series with $p$-adic coefficients.
Its radius of convergence is equal to~$1$.

The sequence~$(F_n)$ satisfies Hypotheses~(1)
and~(3) of Proposition~\ref{prop.3}.
Moreover, $F_n$ converges to~$f$, uniformly on any compact
subset of the open disk~$D(0,1)$.
However, any limit measure~$\nu$
of the sequence~$\nu(F_n)$ satisfies $\nu\geq\frac12\delta_1$.
In particular, $\nu\neq \delta_{\xi(1)}$, so that $(F_n)$
does not satisfies the conclusion of Proposition~\ref{prop.3}.

Moreover, for any~$n$, $F_n$ has at least $\frac12 \deg(F_n)$ 
irreducible factors of degree~$1$,
so that the conclusion of Theorem~\ref{exem-3.4} does not hold for~$f$
neither.
\end{exem}

It remains an interesting open question to find more
general hypotheses on a power series~$f$ of given radius of convergence~$R$
which would guarantee that, in adequate subsequences,
the measures~$\nu(f_n)$ equidistribute towards the Gauss point~$\xi(R)$.

\bibliographystyle{amsplain}
\bibliography{aclab,acl,jentzsch}
\end{document}